\documentclass[12pt]{amsart}
\usepackage{amssymb,amsmath,amsthm,latexsym,mathtools,dsfont}
\usepackage{mathrsfs}
\usepackage{mdwlist}
\usepackage{graphicx}
\usepackage{enumerate}
\usepackage[shortlabels]{enumitem}
\usepackage{a4wide}

\usepackage{ifthen}
\usepackage{xargs}

\makeatletter
\newcommand\@dotsep{4.5}
\def\@tocline#1#2#3#4#5#6#7{\relax
  \ifnum #1>\c@tocdepth 
  \else
    \par \addpenalty\@secpenalty\addvspace{#2}%
    \begingroup \hyphenpenalty\@M
    \@ifempty{#4}{%
      \@tempdima\csname r@tocindent\number#1\endcsname\relax
    }{%
      \@tempdima#4\relax
    }%
    \parindent\z@ \leftskip#3\relax
    \advance\leftskip\@tempdima\relax
    \rightskip\@pnumwidth plus1em \parfillskip-\@pnumwidth
    #5\leavevmode\hskip-\@tempdima #6\relax
    \leaders\hbox{$\m@th
      \mkern \@dotsep mu\hbox{.}\mkern \@dotsep mu$}\hfill
    \hbox to\@pnumwidth{\@tocpagenum{#7}}\par
    \nobreak
    \endgroup
  \fi}
\makeatother 

\let\oldtocsection=\tocsection
\let\oldtocsubsection=\tocsubsection

\renewcommand{\tocsection}[2]{\hspace{0em}\oldtocsection{#1}{#2}}
\renewcommand{\tocsubsection}[2]{\hspace{22pt}\oldtocsubsection{#1}{#2}}

\newcommandx{\yaHelper}[2][1=\empty]{%
\ifthenelse{\equal{#1}{\empty}}%
  { \ensuremath{ \scriptstyle{ #2 } } } 
  { \raisebox{ #1 }[0pt][0pt]{ \ensuremath{ \scriptstyle{ #2 } } } }  
}   
\newcommandx{\yrightarrow}[4][1=\empty, 2=\empty, 4=\empty, usedefault=@]{%
  \ifthenelse{\equal{#2}{\empty}}
  { \xrightarrow{ \protect{ \yaHelper[ #4 ]{ #3 } } } } 
  { \xrightarrow[ \protect{ \yaHelper[ #2 ]{ #1 } } ]{ \protect{ \yaHelper[ #4 ]{ #3 } } } } 
}

\input xy
\xyoption{all}
\usepackage{physics}

\usepackage{xcolor}
\definecolor{darkgreen}{RGB}{0, 153, 51}
\definecolor{violet}{RGB}{112, 73, 170}
\definecolor{darkred}{RGB}{153, 0, 0}
\definecolor{darkdarkblue}{RGB}{0, 0, 102}
\definecolor{darkblue}{RGB}{153, 204, 255}
\definecolor{bluee}{RGB}{204, 230, 255}
\definecolor{bluee2}{RGB}{128, 191, 255}
\definecolor{shadow}{RGB}{82, 122, 122}

\definecolor{my_green1}{RGB}{0, 153, 0}
\definecolor{my_orange}{RGB}{255, 204, 0}
\definecolor{my_yellow}{RGB}{255, 255, 102}

\definecolor{green_pick}{RGB}{51,102,0}
\definecolor{blue_pick}{RGB}{51,51,204}
\definecolor{red_pick}{RGB}{128,0,0}

\usepackage{tikz}
\usepackage{pgfplots, tikz-3dplot}
\usetikzlibrary{shapes.geometric, arrows.meta, decorations.markings}
\pgfplotsset{compat=1.16}
\usetikzlibrary{cd}
\usetikzlibrary{calc}
\usetikzlibrary{arrows, decorations.pathmorphing}
\usetikzlibrary{hobby}
\usetikzlibrary{patterns, intersections}

\makeatletter
\tikzoption{canvas is plane}[]{\@setOxy#1}
\def\@setOxy O(#1,#2,#3)x(#4,#5,#6)y(#7,#8,#9)%
{\def\tikz@plane@origin{\pgfpointxyz{#1}{#2}{#3}}%
\def\tikz@plane@x{\pgfpointxyz{#4}{#5}{#6}}%
\def\tikz@plane@y{\pgfpointxyz{#7}{#8}{#9}}%
\tikz@canvas@is@plane
}

\makeatletter
\newcommand{\gettikzxy}[3]{%
  \tikz@scan@one@point\pgfutil@firstofone#1\relax
  \edef#2{\the\pgf@x}%
  \edef#3{\the\pgf@y}%
}
\makeatother

\newcommand{\R}{\mathbb{R}}

\newcommand{\N}{\mathbb{N}}
\newcommand{\Z}{\mathbb{Z}}
\newcommand{\C}{\mathbb{C}}

\newcommand{\Aa}{\mathcal{A}}
\newcommand{\BB}{\mathcal{B}}

\newcommand{\Dd}{\mathcal{D}}

\newcommand{\Hh}{\mathcal{H}}
\newcommand{\KK}{\mathcal{K}}

\newcommand{\QQ}{\mathcal{Q}}

\newcommand{\UU}{\mathscr{U}}

\newcommand{\ii}{\mathrm{i}}

\newcommand{\p}{\varphi}
\newcommand{\e}{\varepsilon}

\newcommand{\oo}{\overline}
\newcommand{\MM}{\mathbb{M}}

\newcommand{\SOT}{\scalebox{0.82}{SOT}}

\newcommand{\RE}{\mathrm{Re}}

\newcommand{\sot}{\mathop{\scalebox{0.82}{\mbox{{\rm SOT-}}}\hspace*{-1pt}\lim}}

\newcommand{\n}[1]{\|#1\|}
\newcommand{\nn}[1]{{\vert\kern-0.25ex\vert\kern-0.25ex\vert #1 
    \vert\kern-0.25ex\vert\kern-0.25ex\vert}}
\newcommand{\lnn}[1]{{\left\vert\kern-0.25ex\left\vert\kern-0.25ex\left\vert #1 
    \right\vert\kern-0.25ex\right\vert\kern-0.25ex\right\vert}}

\newcommand{\ccup}{\scalebox{0.85}{$\bigcup$}}

\renewcommand{\leq}{\leqslant}
\renewcommand{\geq}{\geqslant}

\newcommand{\cs}{{\rm C}$^\ast$}

\newcommand{\Ext}{\mathrm{Ext}}

\newcommand{\usim}{\,\raise.17ex\hbox{$\scriptstyle\mathtt{\sim}$}}
\renewcommand{\ip}[2]{\langle #1,#2\rangle}

\newtheorem{theorem}{Theorem}[section]
\newtheorem{lemma}[theorem]{Lemma}
\newtheorem{proposition}[theorem]{Proposition}
\newtheorem{corollary}[theorem]{Corollary}
\newtheorem*{theoremA}{Theorem A}
\newtheorem*{theoremB}{Theorem B}
\newtheorem*{theoremC}{Theorem C}

\theoremstyle{definition}
\newtheorem{definition}[theorem]{Definition}

\theoremstyle{remark}
\newtheorem{remark}[theorem]{Remark}
\newtheorem{example}[theorem]{Example}

\numberwithin{equation}{section}

\title[Automatic continuity of operator semigroups]{Automatic continuity of operator\\ semigroups in the Calkin algebra}
\author{Tomasz Kochanek}
\address{Institute of Mathematics, University of Warsaw, Banacha~2, 02-097 Warsaw, Poland}
\email{tkoch@mimuw.edu.pl}
\thanks{This work has been supported by the National Science Centre grant no. 2020/37/B/ST1/01052.}

\begin{document}
\begin{abstract}
We study operator semigroups in the Calkin algebra $\QQ(\Hh)$, represented as a~subalgebra of the algebra of bounded linear operators on a~Hilbert space via one of `canonical' Calkin's representations. Using the BDF theory, we associate with any normal $C_0$-semigroup $(q(t))_{t\geq 0}$ in $\QQ(\Hh)$ an~extension $\Gamma\in\Ext(\Delta)$, where $\Delta$ is the inverse limit of certain compact metric spaces defined purely in terms of the spectrum $\sigma(A)$ of the generator of $(q(t))_{t\geq 0}$. Then we show that, in natural circumstances, if $(q(t))_{t\geq 0}$ is continuous in the strong operator topology, then it is actually uniformly continuous, although there are $C_0$-semigroups in $\QQ(\Hh)$ that are not uniformly continuous.
\end{abstract}
\maketitle


\section{Introduction and summary of the main results}
\noindent
Let $\Hh$ be an infinite-dimensional separable Hilbert space and $\QQ(\Hh)$ be the corresponding Calkin algebra, that is, the quotient $\BB(\Hh)/\KK(\Hh)$ of the algebra of bounded linear operators on $\Hh$ by the ideal of compact operators. The goal of this paper is to investigate strong and uniform continuity of operator semigroups in $\QQ(\Hh)$. Continuity in the strong operator topology is understood here via certain canonical Calkin's representations of $\QQ(\Hh)$ on a~nonseparable Hilbert space. On one hand, we characterize the situation where a~semigroup $(q(t))_{t\geq 0}\subset\QQ(\Hh)$ is a~$C_0$-semigroup in terms of an~easy topological property of a~certain compact metric space naturally associated with $(q(t))_{t\geq 0}$. On the other hand, we provide some automatic continuity-type results also under natural topological conditions. 

Evidently, both automatic continuity (see the classical monograph \cite{dales}) and lifting problems (see the recent book \cite{farah}) play a~vital role in theory of Banach algebras. Our approach is heavily based on the theory of extensions of compact metric spaces developed by Brown, Douglas and Fillmore in their famous work \cite{BDF} where they solved the lifting problem for essentially normal operators. One of our basic results (Proposition~\ref{P_spectrum}) says that with every normal $C_0$-semigroup in $(q(t))_{t\geq 0}\subset \QQ(\Hh)$ one can associate an~extension $\Gamma\in\Ext(\Delta)$ of a~certain inverse limit $\Delta=\varprojlim\Omega_n$ of compact metric spaces $\Omega_n$. More precisely, if $A$ is the infinitesimal generator, then the spectrum of the \cs-algebra $\mathrm{C}^\ast(q(2^{-n}),1_{\QQ(\Hh)})$ is homeomorphic to the inverse limit $\Delta=\varprojlim\{\Omega_n,p_n\}$, where $p_n(z)=z^2$ and
$$
\Omega_n=\oo{\exp(2^{-n}\sigma(A))}\qquad (n=0,1,2,\ldots).
$$

This procedure suggests to distinguish a~class of compact metric spaces arising as inverse limits of the same form as above, and call them {\it admissible}. The process described above can be reversed in the sense that every admissible space naturally induces a~dyadic semigroup $(q(t))_{t\in\mathbb{D}}$ in $\QQ(\Hh)$. ($\mathbb{D}$ is the set of positive dyadic rationals.)

Our first announced main result reads as follows.

\begin{theoremA}
Let $X=\varprojlim X_n$ be an admissible compact metric space \vspace*{-2pt}with $\pi_n\colon X\to X_n$ standing for the $n^{\mathrm{th}}$ projection, for each $n\in\N$. Let also $(q(t))_{t\in\mathbb{D}}$ be the dyadic semigroup induced by the zero element of $\Ext(X)$. Then, the following assertions are equivalent:

\vspace*{2mm}
\begin{enumerate}[label={\rm (\roman*)}, leftmargin=26pt]
\setlength{\itemsep}{2pt}

\item $(q(t))_{t\in\mathbb{D}}$ is strongly continuous with respect to a~fixed Calkin's representation;

\item $(q(t))_{t\in\mathbb{D}}$ is strongly continuous with respect to all Calkin's representations;

\item $\lim_{n\to\infty}\pi_n(\xi)=1$ for every $\xi\in X$.
\end{enumerate}
\end{theoremA}

In order to formulate our results on automatic continuity we introduce a~rather technical geometric condition which is, however, not difficult to verify in concrete situations. Namely, given any closed set $Z\subset\C$ with 
\begin{equation}\label{Z_boundaries}
-\infty<\eta=\inf_{z\in Z}\RE\,z\leq\sup_{z\in Z}\RE\,z=\zeta<+\infty
\end{equation}
we define the vertical sections $S_t=\{u\in\R\colon t+\ii u\in Z\}$ and, for any $t\in [\eta,\zeta]$ and $n\in\N$, consider the following condition:
$$
\big(S_t+2^{n+1}\pi\Z\big)\cap\big(S_t-2^n\pi+2^{n+1}\pi\Z   \big)\neq\varnothing\leqno(\star)
$$
and define
$$
M_t=\{n\in\N\colon \mathrm{(}\!\star\!\mathrm{)}\mbox{ holds true}\,\}.
$$
(Notice that $n\in M_t$ if and only if there is an~odd integer $k$ with $2^nk\pi\in S_t-S_t$.) Then, our results read as follows.

\begin{theoremB}
Let $(q(t))_{t\geq 0}\subset\QQ(\Hh)$ be a $C_0$-semigroup of normal operators with generator $A$, where the strong continuity is understood with respect to Calkin's representation, and let $Z=\sigma(A)$. Assume that $Z$ satisfies the following hypotheses:

\vspace*{2mm}
\begin{itemize}[leftmargin=34pt]
\setlength{\itemsep}{3pt}
\item[{\rm (H$_1$)}] $\exp(2^{-n}Z)$ is closed for sufficiently large $n\in\N$; 

\item[{\rm (H$_2$)}] the set $\ccup_{\eta\leq t\leq\zeta}\,M_t$ is finite,
\end{itemize}
and that the induced admissible compact metric space $X$ satisfies $\Ext(X)=0$. Then $(q(t))_{t\geq 0}$ is uniformly continuous.
\end{theoremB}

Another version of the above automatic continuity result, which does not require any condition on the extension group, reads as follows.

\begin{theoremC}
Let $(Q(t))_{t\geq 0}\subset\BB(\Hh)$ be a $C_0$-semigroup of normal operators with generator $T$ such that $\sigma_{\mathrm{ess}}(T)=\sigma(T)$, and let $q(t)=\pi Q(t)$ for $t\geq 0$. If the semigroup $(q(t))_{t\geq 0}$ is strongly continuous with respect to Calkin's representation, $A$ is its generator, and conditions {\rm (H$_1$)} and {\rm (H$_2$)} are satisfied for $Z=\sigma(A)$, then $(q(t))_{t\geq 0}$ is uniformly continuous. 
\end{theoremC}

It is worth mentioning that the characterization given in Theorem~A allows us to construct an~example of a~$C_0$-semigroup (with respect to all Calkin's representations) in $\QQ(\Hh)$ that is not uniformly continuous (see Example~\ref{c0not}).

\section{Preliminaries and tools}
\subsection{Connections with the BDF theory}

Here, we shall collect some necessary basic facts from the Brown--Douglas--Fillmore theory of extensions, and explain its relation to our results. For more details on the BDF theory, the reader is referred to the seminal paper \cite{BDF} or to nice expositions of that theory in \cite[Ch.~VII]{blackadar-K} or \cite[Ch.~IX]{davidson}.

Let $X$ be a compact metric space. By an~{\it extension} of $C(X)$ ({\it by} $\KK(\Hh)$) we mean any pair $(\Aa,\p)$, where $\Aa$ is a~\cs-subalgebra of $\BB(\Hh)$ containing the compact operators and the~identity operator, and $\p\colon\Aa\to C(X)$ is a~$^\ast$-homomorphism such that
\begin{equation*}
\begin{tikzcd}
0 \arrow[r] & \KK(\Hh) \arrow[r, "\iota"] & \mathcal{A} \arrow[r, "\p"] & C(X) \arrow[r] & 0
\end{tikzcd}
\end{equation*}
is an exact sequence, where $\iota$ is the inclusion map. To every extension one can associate the~so-called {\it Busby invariant} which is a~unital $^\ast$-monomorphism $\tau\colon C(X)\to\QQ(\Hh)$ defined as $\tau=\pi\p^{-1}$, which is the inverse of the~identification $\mathcal{A}/\KK(\Hh)\cong C(X)$. Conversely, any such $^\ast$-monomorphism gives rise to an~extension $(\pi^{-1}\tau(C(X)),\tau^{-1}\pi)$. In this setting, two extensions of $C(X)$ are called {\it equivalent} if the associated Busby invariants $\tau_1$ and $\tau_2$ satisfy $\tau_2=\pi(U)^\ast\tau_1\pi(U)$ for some unitary $U\in\BB(\Hh)$. We write $[\tau]$ for the extension (equivalence class) generated by a~unital $^\ast$-monomorphism $\tau\colon C(X)\to\QQ(\Hh)$.

Of fundamental importance is the fact that the~collection $\Ext(X)$ of all equivalence classes of extensions of $C(X)$ forms a~group when equipped with an~operation $+$ defined in terms of $^\ast$-monomorphisms $C(X)\to\QQ(\Hh)$ as $[\tau_1]+[\tau_2]=[\tau_1\oplus\tau_2]$. Here, we use an~isomorphism $\Hh\oplus\Hh\cong\Hh$ which allows us to identify the matrix algebra $\MM_2(\QQ(\Hh))$ with $\QQ(\Hh)$, as $\MM_2(\KK(\Hh))$ is mapped onto $\KK(\Hh)$. The zero element of that group can be constructed as follows. Take any infinite direct sum decomposition $\Hh=\bigoplus_{i=1}^\infty\Hh_i$, where each $\Hh_i$ is infinite-dimensional, pick a~countable dense subset $\{\xi_{i}\colon i\in\N\}$ of $X$ and define $\sigma\colon C(X)\to\BB(\Hh)$ by 
\begin{equation}\label{sigma_neutral}
\sigma(g)=\bigoplus_{i=1}^\infty g(\xi_{i})I_i,
\end{equation}
where $I_i$ is the identity operator on $\Hh_i$. Plainly, there are no nonzero compact operators in the range of $\sigma$, which implies that $\pi\sigma\colon C(X)\to\QQ(\Hh)$ is a~$^\ast$-monomorphism admitting the section $\sigma$, hence it determines the trivial extension of $C(X)$. That all trivial extensions are equivalent follows from the celebrated Weyl-von~Neumann--Berg theorem; see \cite{berg} and \cite[Thms.~II.4.6 and IX.2.1]{davidson}. Throughout this paper, we denote by $\Theta\in\Ext(X)$ the zero element in the~group of extensions of $X$, i.e. $\Theta=[\sigma]$, where $\sigma$ is given as in \eqref{sigma_neutral}. The property that each $[\tau]\in\Ext(X)$ has an~inverse, that is, there exists a~Busby invariant $\sigma$ with $[\tau\oplus\sigma]=\Theta$, is based on the~lifting property of positive unital maps on $C(X)$ and Naimark's dilation theorem (see \cite[Thm.~IX.5.1]{davidson}).

Given two compact metric spaces $X$ and $Y$, and a~continuous map $f\colon X\to Y$, there is an~induced map $f_\ast\colon\Ext(X)\to\Ext(Y)$ defined as
$$
f_\ast(\tau)(g)=\tau(g\circ f)\oplus \sigma(g)\quad (g\in C(Y)),
$$
where $\sigma$ is any $^\ast$-monomorphism corresponding to the trivial extension of $C(Y)$. We add the second direct summand in order to guarantee that the resulting map $f_\ast(\tau)$ is injective. One can verify that $(fg)_\ast=f_\ast g_\ast$ whenever these compositions make sense.

Recall that for any compact metric space $X$, the {\it cone} $CX$ over $X$ is obtained from $X\times I$ by collapsing $X\times\{0\}$ to a~single point, where $I=[0,1]$. The {\it suspension} $SX$ is obtained from $X\times I$ by collapsing $X\times\{0\}$ and $X\times\{1\}$ to two distinct points.

The extension functor is defined for ranks $q\leq 1$ by $\Ext_q(X)=\Ext(S^{1-q}X)$. It was shown in \cite[\S 6]{BDF} that, analogously to Bott's periodicity in $K$-theory, there exist isomorphisms
$$
\mathrm{Per}_\ast\colon \Ext_{q-2}(X)\xrightarrow[\phantom{xx}]{}\Ext_q(X)\quad (r\leq 1).
$$
This allows us to extend the definition of $\Ext$ to all integer dimensions:
$$
\Ext_q(X)=\left\{\begin{array}{ll}
\Ext(X) & \mbox{if }q\mbox{ is odd},\\
\Ext(SX) & \mbox{if }q\mbox{ is even}.
\end{array}\right.
$$
Since every continuous map $f\colon X\to Y$ naturally induces a~map $Sf\colon SX\to SY$ by appropriately quotienting $f\times\mathrm{id}\colon X\times I\to Y\times I$, there is an~induced homomorphism $f_\ast\colon \Ext_q(X)\to\Ext_q(Y)$, for any $q\in\Z$. We also define extensions of compact pairs by the obvious formula $\Ext_q(X,A)=\Ext_q(X/A)$. Therefore, for any admissible map $f\colon (X,A)\to (Y,B)$ between compact pairs (i.e. $f$ is continuous and $f(A)\subseteq B$), there is an~induced homomorphism $f_\ast\colon \Ext_q(X,A)\to\Ext_q(Y,B)$.

Suppose $\{X_n,p_n\}_{n=0}^\infty$ is an~inverse system of compact metric spaces. Let $X=\varprojlim X_n$ and $q_n\colon X\to X_n$ stand for the coordinate maps, for $n\in\N_0$, so that $p_nq_{n+1}=q_n$. Hence, we have another inverse system of groups $\{\Ext(X_n),p_{n\ast}\}_{n=0}^\infty$. Since $p_{n\ast}q_{(n+1)\ast}=q_{n\ast}$, we can define an~{\it induced map} 
$$
P\colon \Ext(X)\to\varprojlim\Ext(X_n),\quad P(\tau)=(q_{n\ast}\tau)_{n=0}^\infty. 
$$
By \cite[Thm.~8.4]{BDFu}, the induced map is always surjective, but in general not injective. However, Milnor \cite{milnor} showed that for any homology theory satisfying the Steenrod axioms, except the dimension axiom (see \cite{ES}), one can build an~exact sequence which measures the lack of continuity of the $\Ext$-functor with respect to inverse limit (see \cite[\S 5]{Dou}).
\begin{theorem}[see {\cite[Thm.~4]{milnor}} and {\cite[Cor.~7.4]{BDF}}]\label{milnor_thm}
For any inverse system $\{X_n\}$ of compact metric spaces, and any $k\in\Z$, there exists an~exact sequence
\begin{equation*}
\begin{tikzcd}
0 \arrow[r] & \varprojlim{}^{(1)}\Ext_{k+1}(X_n) \arrow[r] & \Ext_k(\varprojlim X_n) \arrow[r, "P"] & \varprojlim\Ext_k(X_n) \arrow[r] & 0
\end{tikzcd}
\end{equation*}
where $\varprojlim{}^{(1)}$ is the first derived functor of inverse limit.
\end{theorem}

\subsection{Admissible compact metric spaces}
In this paper, we shall be concerned with extensions of some special projective limits of compact subsets of the complex plane. Recall that the projective (inverse) limit of an~inverse system $\{X_n,f_n\}_{n\geq 0}$, that is, a~sequence of topological spaces and continuous maps $f_n\colon X_{n+1}\to X_n$, is defined as
$$
\varprojlim X_n=\Big\{\mathbf{x}=(x_n)_{n=0}^\infty\in \prod_{n=0}^\infty X_n\colon f_n(x_{n+1})=x_n\,\,\,\mbox{for }n\geq 0\Big\}.
$$
We denote by $\pi_n\colon\varprojlim X_n\to X_n$ the projection onto the~$n^{\mathrm{th}}$ coordinate, that is, $\pi_n(\mathbf{x})=x_n$ for $\mathbf{x}=(x_n)_{n=0}^\infty$. The topology on $\varprojlim X_n$ is the weakest topology under which all the maps $\pi_n$ are continuous. 

As we will see, certain inverse limits of compact metric spaces arise naturally when considering operator semigroups indexed by positive dyadic rational numbers. We denote by $\mathbb{D}$ the set of such numbers, i.e. $\mathbb{D}=\{k2^{-m}\colon k,m\in\N\}$. The following observation will be used several times in the sequel. Namely, if $(q(2^{-n}))_{n\geq 0}$ is a~sequence of elements of any \cs-algebra $\mathcal{A}$ such that \vspace*{-1pt} $q(2^{-(n+1)})^2=q(2^{-n})$ for each $n=0,1,2,\ldots,$ then for any $t\in\mathbb{D}$ written in the form $t=t_0+\sum_{i=1}^j 2^{-m_i}$ with integers $t_0,j\geq 0$ and $1\leq m_1<\ldots<m_j$, we define $q(t)=q(1)^{t_0}q(2^{-m_1})\cdot\ldots\cdot q(2^{-m_j})$. It is then readily seen that $q(s+t)=q(s)q(t)$ for all $s,t\in\mathbb{D}$, hence we have extended $(q(2^{-n}))_{n\geq 0}$ to a~semigroup $(q(t))_{t\in\mathbb{D}}$, to which we shall refer as a~{\it dyadic semigroup}. Obviously, such an~extension is unique if we want to preserve the semigroup property. Given a~faithful $^\ast$-representation $\gamma$ of $\mathcal{A}$ on a~Hilbert space $\mathbb{H}$, we call $(q(t))_{t\in\mathbb{D}}$ a~$C_0$-{\it semigroup}, provided that there exists a~$C_0$-semigroup $(T(t))_{t\geq 0}\subset\BB(\mathbb{H})$ such that $T(t)=\gamma(q(t))$ for every $t\in\mathbb{D}$. In particular, if $(Q(t))_{t\in\mathbb{D}}\subset\BB(\Hh)$ is a~dyadic semigroup, we say that $(Q(t))_{t\in\mathbb{D}}$ is a~$C_0$-semigroup if it can be extended to a~$C_0$-semigroup in $\BB(\Hh)$ (in the usual sense) defined on $[0,\infty)$. This can be done if and only if $(Q(t))_{t\in\mathbb{D}}$ is \SOT-continuous and uniformly bounded on every bounded set of dyadic rationals (see Remark~\ref{cont1_R} below).

Anticipating our considerations presented in the next sections we distinguish the following class of compact metric spaces.

\begin{definition}\label{admissible_D}
We call a compact metric space $X$ {\it admissible} if $X=\varprojlim\{X_n,f_n\}_{n\geq 0}$ for some compact subsets $X_n$ of $\C$ being of the form
\begin{equation}\label{adm_Z_D}
X_n=\oo{\exp(2^{-n}Z)}\qquad (n=0,1,2,\ldots),
\end{equation}
where $Z\subset\C$ is a~fixed closed set such that $\sup_{z\in Z}\mathrm{Re}\,z<\infty$, and the connecting maps are all given by $f_n(z)=z^2$. 
\end{definition}

The main starting point for our considerations is the fact that every normal $C_0$-semigroup in $\QQ(\Hh)$ gives rise to an~extension of an~admissible compact metric space which is completely described in terms of the spectrum of the infinitesimal generator. This will be proved in details in Section~3 (see Proposition~\ref{P_spectrum}). For now, we note that with every such extension one can naturally associate a~dyadic operator semigroup in $\QQ(\Hh)$.

\begin{definition}\label{ind_D}
Let $X$ be an admissible compact metric space and let $\Gamma\in\Ext(X)$ be the extension induced by a~Busby invariant $\tau\colon C(X)\to\QQ(\Hh)$. We say that $\Gamma$ {\it induces a~dyadic semigroup} $(q(t))_{t\in\mathbb{D}}\subset\QQ(\Hh)$, provided that $(q(t))_{t\in\mathbb{D}}$ is the unique semigroup extension of $(q(2^{-n}))_{n\geq 0}$ given by 
$$
q(2^{-n})=\tau(\pi_n)\quad\,\, (n=0,1,2,\ldots),
$$

\vspace*{1mm}\noindent
where $\pi_n\in C(X)$ is the projection onto the $n^{\mathrm{th}}$ coordinate.
\end{definition}

\begin{remark}\label{rem_exp}
If $X$ is an inverse limit of sets $X_n\subset\C$ given by \eqref{adm_Z_D}, then since any Busby invariant (being a~$^\ast$-monomorphism) is an~isometry, the dyadic semigroup $(q(t))_{t\in\mathbb{D}}$ induced by an~arbitrary element of $\Ext(X)$ satisfies the estimate
$$
\n{q(2^{-n})}\leq\exp(2^{-n}\zeta)\quad\,\, (n=0,1,2,\ldots),
$$
where $\zeta=\sup_{z\in Z}\mathrm{Re}\,z$.
\end{remark}

\begin{definition}\label{C0_D}
For an admissible compact metric space $X$ and a~faithful $^\ast$-representation $\gamma\colon\QQ(\Hh)\to\BB(\mathbb{H})$, we define $\Ext_{C_0,\gamma}(X)\subseteq\Ext(X)$ to be the set of extensions inducing $C_0$-semigroups in $\QQ(\Hh)$.
\end{definition}

\begin{remark}\label{C0_correct_R}
The above definition is correctly posed in the sense that the question whether a~given $[\sigma]\in\Ext(X)$ induces a~$C_0$-semigroup does not depend on the choice of representative. For, suppose $U\in\BB(\Hh)$ is unitary and $\tau=\pi(U)^\ast\sigma\pi(U)$. Then, for any $\mathbf{x}\in\mathbb{H}$, $m,n\in\N$, and any complex polynomials $P,P_m\in\C[z]$, we have
\begin{equation*}
\begin{split}
    \gamma\big[\pi(U)^\ast\sigma(P_m\circ\pi_m) \pi(U)\big]\mathbf{x}&-\gamma(\tau(P\circ\pi_n))\mathbf{x}\\
    &=\gamma\big[\pi(U)^\ast\sigma(P_m\circ\pi_m-P\circ\pi_n)\pi(U)\big]\mathbf{x}\xrightarrow[\,m\to\infty\,]{}0,
\end{split}
\end{equation*}
provided that $\sigma$ induces a~$C_0$-semigroup and $P_m\circ\pi_m$ converge pointwise on $X$ to $P\circ\pi_n$. This shows that $\tau$ generates a~$C_0$-semigroup if and only if so does $\sigma$. We have used here an~observation that each element of the induced dyadic semigroup corresponds, as in Definition~\ref{ind_D}, to a~complex polynomial on the $n^{\mathrm{th}}$ coordinate of $X$, for some $n\in\N$ (see also the proof of Lemma~\ref{lifting_L}).
\end{remark}

\subsection{Calkin representations}

We will be now dealing with $C_0$-semigroups in $\QQ(\Hh)$ with respect to some concrete representations of $\QQ(\Hh)$ as subalgebras of $\BB(\mathbb{H})$ for a~Hilbert space $\mathbb{H}$. The original construction by Calkin \cite{calkin} goes as follows.

Let $\UU$ be any nonprincipal ultrafilter on $\N$ which we keep fixed for the rest of the paper. The limit along $\UU$ gives rise to a~positive functional $\mathrm{LIM}_\UU\in\ell_\infty^\ast$ which satisfies $\liminf_n a_n\leq\mathrm{LIM}_{n,\UU}a_n\leq\limsup_n a_n$ for every real valued $(a_n)\in\ell_\infty$. Let $\mathscr{W}$ be the collection of all weakly null sequences in $\Hh$ on which we consider an~equivalence relation $(x_n)\sim (y_n)$ defined by $\lim_{n,\UU}\n{x_n-y_n}=0$. We write $\mathscr{W}^{\usim}$ for the collection of all equivalence classes and $[(x_n)]_{\sim}$ for the equivalence class of a~sequence $(x_n)\in\mathscr{W}$. The formula
$$
\big\langle [(x_n)]_\sim, [(y_n)]_\sim\big\rangle=\mathop{\mathrm{LIM}}\limits_{n,\UU}(\ip{x_n}{y_n})_{n=1}^\infty
$$
gives rise to an inner product in $\mathscr{W}^\sim$. Notice that, by the Cauchy--Schwarz inequality,  the right-hand side does not depend on the choice of representatives. Also, if $(x_n)$ is not a~null sequence, then $\ip{[(x_n)]_\sim}{[(x_n)]_\sim}>0$. 

Let $\mathbb{H}$ be the completion of $\mathscr{W}^{\usim}$ under the norm $\n{[(x_n)]_\sim}=\mathrm{LIM}_{n,\UU}\n{x_n}$. Then, $\mathbb{H}$ is a~Hilbert space of density $\mathfrak{c}$ such that $\QQ(\Hh)$ can be faithfully represented on $\BB(\Hh)$ in the following way. For any $T\in\BB(\Hh)$, define $\Phi_0(T)$ to be the linear operator on $\mathscr{W}^{\usim}$ given by $\Phi_0(T)[(x_n)]_\sim=[(Tx_n)]_\sim$. Since $\n{\Phi_0T}\leq\n{T}$, there is a~unique extension of $\Phi_0(T)$ to a~bounded linear operator on $\mathbb{H}$ which we denote by $\Phi(T)$. Obviously, the map $T\mapsto\Phi(T)\in\BB(\mathbb{H})$ is a~$^\ast$-homomorphism and since $\mathrm{ker}\,\Phi=\KK(\Hh)$, we can define an~induced map
$$
\gamma\colon\QQ(\Hh)\to\BB(\mathbb{H}),\,\,\, \gamma (\pi(T))=\Phi(T)\quad (T\in\BB(\Hh)),
$$
which is a~faithful representation of $\QQ(\Hh)$. In what follows, we shall call $\gamma$ the {\it Calkin representation}. If we want to stress that it comes from the ultrafilter $\UU$, we call it the $\UU$-{\it Calkin representation}. A~thorough study of Calkin's representations was done by Reid in \cite{reid}, whereas other, more subtle representations, with type II$_\infty$ factor as the range, were constructed by Anderson and Bunce (see \cite{anderson} and \cite{AB}).

The next technical lemma refers to the well-known criterion for \SOT-continuity of an~operator semigroup $(T(t))_{t\geq 0}$ on a~Banach space $X$, which requires $T(t)$ to be uniformly bounded on an~interval $[0,\delta]$ and converge in norm pointwise on a~dense subset of $X$ (see \cite[Prop.~1.3]{EN}).

\begin{lemma}\label{cont1_L}
Let $X$ be an admissible compact metric space and let $\Gamma\in\Ext(X)$ induce a~dyadic semigoup $(q(t))_{t\in\mathbb{D}}\subset\QQ(\Hh)$. Let also $\gamma$ be a~faithful $^\ast$-representation of $\QQ(\Hh)$ on a~Hilbert space $\mathbb{H}$. Assume that for some dense set $D\subset\mathbb{H}$ we have 
$$
\lim_{t\in\mathbb{D},\,t\to 0}\gamma (q(t))\mathbf{x}=\mathbf{x}\quad\mbox{for each }\, \mathbf{x}\in D.
$$
Then $\sot_{t\in\mathbb{D},\, t\to 0}\gamma(q(t))=I_{\mathbb{H}}$ and there is a~$C_0$-semigroup $(T(t))_{t\geq 0}\subset\BB(\mathbb{H})$ which extends $(\gamma(q(t)))_{t\in\mathbb{D}}$.
\end{lemma}
\begin{proof}
For any sequence $(t_n)_{n=1}^\infty\subset\mathbb{D}$ with $t_n\to 0$, consider the compact set $K=\{0\}\cup\{t_n\colon n\in\N\}$ and notice that the map $K\ni t\mapsto \gamma (q(t))$, where $\gamma(q(0))=I_{\mathbb{H}}$, is norm bounded and such that $K\ni t\mapsto \gamma(q(t))\mathbf{x}$ is continuous on $K$ for $\mathbf{x}\in D$. Fix $\mathbf{y}\in\mathbb{H}$, $\e>0$, and pick $\mathbf{x}\in D$, $\delta>0$ and $M>0$ such that $\n{\mathbf{x}-\mathbf{y}}<\e$, $\n{\gamma(q(t))\mathbf{x}-\mathbf{x}}<\e$ for $t\in K$ with $t<\delta$, and $\n{\gamma(q(t))}\leq M$ for every $t\in K$. Then, by triangle inequality, $\n{\gamma(q(t))\mathbf{y}-\mathbf{y}}<(M+2)\e$ for each $t\in K$, $t<\delta$, which shows that $\sot_{n\to\infty}\gamma(q(t_n))=I_{\mathbb{H}}$ and hence $\sot_{t\in\mathbb{D},\,t\to 0}\gamma(q(t))=I_{\mathbb{H}}$.

Fix any $\mathbf{x}\in\mathbb{H}$ and notice that in view of Remark~\ref{rem_exp}, for any $s,t\in\mathbb{D}$, $s<t$, we have
$$
\n{\gamma(q(t))\mathbf{x}-\gamma(q(s))\mathbf{x}}\leq\n{q(s)}\!\cdot\!\n{\gamma(q(t-s))\mathbf{x}-\mathbf{x}}\leq e^{s\zeta}\n{\gamma(q(t-s))\mathbf{x}-\mathbf{x}}.
$$
Hence, the map $\mathbb{D}\ni t\mapsto \gamma(q(t))\mathbf{x}$ is uniformly continuous on $[0,T]\cap\mathbb{D}$ for any $T>0$. We denote its continuous extension to the whole of $[0,\infty)$ by $T(\cdot)\mathbf{x}$, so that for every $\mathbf{x}\in\mathbb{H}$ we have a~continuous function $[0,\infty)\ni t\mapsto T(t)\mathbf{x}$. Since for any $t\in [0,\infty)\setminus\mathbb{D}$ and $\mathbf{x}\in\mathbb{H}$, we have $T(t)\mathbf{x}=\lim_{u\in\mathbb{D},\,u\to t}\gamma(q(u))\mathbf{x}$, the map $\mathbf{x}\mapsto T(t)\mathbf{x}$ is a~bounded linear operator on $\mathbb{H}$. Therefore, $(T(t))_{t\geq 0}$ is a~$C_0$-semigroup in $\BB(\mathbb{H})$ which satisfies $T(t)=\gamma(q(t))$ for $t\in\mathbb{D}$. 
\end{proof}

\begin{remark}\label{cont1_R}
Of course, the same proof as above works also in a~simpler situation, where $(Q(t))_{t\in\mathbb{D}}\subset\BB(\Hh)$ is a~dyadic semigroup satisfying:
\begin{itemize}[leftmargin=24pt]
\setlength{\itemsep}{3pt}
\item $\sup_{t\in [0,s]\cap\mathbb{D}}\n{Q(t)}<\infty$ for every $s>0$;

\item $\sot_{t\in\mathbb{D},\,t\to 0}Q(t)=I.$
\end{itemize}
Then there is a $C_0$-semigroup $(Q(t))_{t\geq 0}\subset\BB(\Hh)$ which extends the given one.
\end{remark}

\section{Extensions generated by semigroups}
\subsection{An inverse limit associated with the spectrum}

As we have already announced, normal $C_0$-semigroups in the Calkin algebra naturally generate extensions of admissible compact metric spaces. Of course, to speak sensibly about $C_0$-semigroups we need to choose a~faithful $^\ast$-representation of $\QQ(\Hh)$ on a~Hilbert space $\mathbb{H}$. In the results of the present section, the choice of representation is arbitrary. 

Below, we can either assume that $(q(t))_{t\geq 0}$ is a~$C_0$-semigroup of normal elements of $\QQ(\Hh)$, or assume a~weaker condition that the~dyadic semigroup $(q(t))_{t\in\mathbb{D}}$ is a $C_0$-semigroup (i.e. it can be extended to a~$C_0$-semigroup $(T(t))_{t\geq 0}\subset\BB(\mathbb{H})$), however, we then need to assume that all $T(t)$ are normal operators. The reason is that we apply the spectral mapping theorem for normal $C_0$-semigroups (see \cite[Cor.~2.12]{EN}).

\begin{proposition}\label{P_spectrum}
Let $(q(t))_{t\geq 0}\subset\QQ(\Hh)$ be a $C_0$-semigroup of normal operators in the Calkin algebra. Let
$$
\Aa_0=\mathrm{C}^\ast\big(\{q(2^{-n})\colon n=\infty,0,1,2,\ldots\}\big)
$$

\vspace*{1mm}\noindent
be the \cs-subalgebra of $\QQ(\Hh)$ generated by the identity and all $q(2^{-n})$ for $n\in\N_0$, and let
$$
\mathcal{E}=\pi^{-1}(\Aa_0)
$$
be the \cs-subalgebra of $\BB(\Hh)$ generated by $\{q(2^{-n})\colon n=\infty,0,1,2,\ldots\}+\KK(\Hh)$.

\vspace*{1mm}
\begin{itemize}[leftmargin=20pt]
\setlength{\itemsep}{2pt}
    \item[{\rm (a)}] Let $A$ be the generator of $(q(t))_{t\geq 0}$ and define
    $$
    \Omega_n=\oo{\exp(2^{-n}\sigma(A))}\quad (n=0,1,2,\ldots)
    $$
Then, $\Aa_0$ is a~commutative \cs-algebra and its maximal ideal space $\Delta$ is homeomorphic to the projective limit of the inverse system $\{\Omega_n,p_n\}_{n\geq 0}$, where $p_n(z)=z^2$ for each $n=0,1,2,\ldots$
    
    \item[{\rm (b)}] The \cs-algebra $\mathcal{E}$ contains $\KK(\Hh)$ as an ideal and there is an~exact sequence
    $$
    \begin{tikzcd}
    0 \arrow[r] & \KK(\Hh) \arrow[r, "\iota"] & \mathcal{E} \arrow[r, "\theta"] & C(\Delta) \arrow[r] & 0,
    \end{tikzcd}
    $$
where $\theta(T)=\widehat{\pi(T)}$ and $\Aa_0\ni q\xmapsto[\phantom{xx}]{}\widehat{q}\in C(\Delta)$ is the Gelfand transform.
\end{itemize}
\end{proposition}

\begin{proof}
(a) First, observe that since there exists $\zeta<\infty$ such that $\mathrm{Re}\,\lambda\leq\zeta$ for each $\lambda\in\sigma(A)$, all the sets $\Omega_n$ are compact subsets of $\C$. Moreover, $A$ is normal and if $\mathsf{E}^A$ stands for the spectral decomposition of $A$, then each $q(t)$ can be calculated via functional calculus in $L_\infty(\mathsf{E}^A)$ by 
$$
q(t)=\int_{\sigma(A)}e^{t\lambda}\,\dd \mathsf{E}^A(\lambda)\qquad (t\geq 0),
$$
as $\vert e^{t\lambda}\vert\leq e^{t\zeta}$ and hence the function under the integral is bounded. Plainly, $q(s), q(t), q(t)^\ast$ commute for all $s,t\geq 0$, thus $\Aa_0$ is commutative.

The joint spectrum of the set $\{q(2^{-n})\colon n=0,1,2,\ldots\}$ is a~compact subset of $\C^\infty$ defined by
$$
\sigma_{\Aa_0}\big(q(2^{-n})\colon n=0,1,2,\ldots\big)=\big\{(\p(q(2^{-n})))_{n=0}^\infty\colon \p\in\Delta\big\}
$$
and the map 
$$
\Delta\ni\p\xmapsto[\phantom{xxx}]{}(\p(q(2^{-n})))_{n=0}^\infty
$$
is a homeomorphism between $\Delta$ and $\sigma_{\Aa_0}(q(2^{-n})\colon n=0,1,2,\ldots)$ (see \cite[Cor.~3.1.13]{rickart}). On the other hand, a~sequence $\boldsymbol{\lambda}=(\lambda_n)_{n=1}^\infty\in\C^\infty$ belongs to $\sigma_{\Aa_0}(q(2^{-n})\colon n=0,1,2,\ldots)$ if and only if
\begin{equation}\label{q_lambda}
q(\boldsymbol{\lambda})\coloneqq \sum_{n=0}^\infty 2^{-n}\frac{(\lambda_n I-q(2^{-n}))^\ast(\lambda_n I-q(2^{-n}))}{\n{\lambda_n I-q(2^{-n})}^2}
\end{equation}
is not invertible in $\QQ(\Hh)$. Indeed, as each summand is a~positive operator, we infer that for every linear multiplicative functional $\p\in\Delta$ we have $\p(q(\boldsymbol{\lambda}))=0$ if and only if $\p(q(2^{-n}))=\lambda_n$ for each $n=0,1,2,\ldots$ Hence, if $q(\boldsymbol{\lambda})$ is not invertible we pick $\p\in\Delta$ so that $\p(q(\boldsymbol{\lambda}))=0$ to see that $\boldsymbol{\lambda}$ belongs to the joint spectrum. Conversely, if $q(\boldsymbol{\lambda})$ is invertible, then we have $\p(q(\boldsymbol{\lambda}))\neq 0$ for every $\p\in\Delta$, thus $\boldsymbol{\lambda}$ is not in the joint spectrum.

Fix any $\boldsymbol{\lambda}=(\lambda_n)_{n=0}^\infty\in\C^\infty$. The operator $(\lambda_n I-q(2^{-n}))^\ast (\lambda_n I-q(2^{-n}))$ corresponds via functional calculus to the map $\phi_n\in L_\infty(\mathsf{E}^A)$ given by
$$
\phi_n(z)=\vert\lambda_n-\exp(2^{-n}z)\vert^2.
$$
For every $z\in\sigma(A)$, we have $\mathrm{Re}\, z\leq\zeta$ and hence
$$
\n{\phi_n}_\infty\leq \big(\abs{\lambda_n}+\exp(2^{-n}\zeta)\big)^2
$$
which implies that each denominator in formula \eqref{q_lambda} is majorized by a~constant and cannot become arbitrarily large after applying functional calculus and varying $z$ over $\sigma(A)$. Hence, $q(\boldsymbol{\lambda})$ is noninvertible if and only if $0$ lies in the closure of the range of the map
$$
\sigma(A)\ni z\xmapsto[\phantom{xxx}]{}\sum_{n=0}^\infty 2^{-n}\frac{\phi_n(z)}{\n{\phi_n}_\infty},
$$
which implies that each $\lambda_n$ must belong to the closure of $\exp(2^{-n}\sigma(A))$ which is denoted by $\Omega_n$. Moreover, for any $n=0,1,2,\ldots$ we can pick $z\in\sigma(A)$ so that both $\phi_n(z)$ and $\phi_{n+1}(z)$ are arbitrarily close to zero. Since $\exp(2^{-n-1}z)^2=\exp(2^{-n}z)$, we infer that for $q(\boldsymbol{\lambda})$ being noninvertible we also must have $\lambda_{n+1}^2=\lambda_n$ ($n=0,1,\ldots$). This means that every element of the joint spectrum belongs to the inverse limit $\varprojlim\Omega_n$.

Conversely, fix any $\boldsymbol{\lambda}\in\varprojlim\Omega_n$ and any $\e>0$. Take $C>1$ such that $\vert \lambda_n+\exp(2^{-n}z)\vert\leq C$ for all $n\in\N_0$ and $z\in\sigma(A)$, and define
$$
W_n=\big\{z\in\sigma(A)\colon \vert \lambda_n-\exp(2^{-n}z)\vert\leq C^{-n}\e\big\}\quad (n=0,1,2,\ldots).
$$
Observe that for $z\in W_n$, we have
\begin{equation*}
    \begin{split}
        \vert \lambda_{n-1}-\exp(2^{-n+1}z)\vert &=\vert\lambda_n^2-\exp(2\cdot 2^{-n}z)\vert\\
        &=\vert \lambda_n-\exp(2^{-n}z)\vert\cdot\vert \lambda_n+\exp(2^{-n}z)\vert\leq C^{-n+1}\e
    \end{split}
\end{equation*}
which shows that $z\in W_{n-1}$ and, similarly, $z\in W_{n-2},\ldots,W_0$. Therefore, $\bigcap_{j=0}^n W_j\neq\varnothing$. This means that the identity map
$$
\mathrm{id}\colon \sigma_{\Aa_0}(q(2^{-n})\colon n=0,1,\ldots)\xrightarrow[\phantom{xxx}]{}\,\varprojlim\Omega_n
$$
has dense range. Thus, it is an~onto homeomorphism, as both topologies are the product topology which is compact and Hausdorff. Consequently, $\Delta$ is homeomorphic to $\varprojlim\Omega_n$.

\vspace*{2mm}\noindent
(b) Of course, $\KK(\Hh)$ forms an ideal in $\mathcal{E}$. For every $T\in\mathcal{E}$, we have $\pi(T)\in\Aa_0$ and each element in $\Aa_0$ is of this form. Hence, the formula $\theta(T)= \skew{4}\widehat{\textrm{\emph{\rule{0ex}{1.3ex}\smash{$\pi(T)$}}}}$ yields a~$^\ast$-homomorphism onto $C(\Delta)$. Obviously, $T\in\mathrm{ker}\,\theta$ if and only if $\pi(T)=0$, i.e. $T\in\KK(\Hh)$.
\end{proof}

\subsection{Lifting problems}

Henceforth, we will be using the notation introduced in Proposition~\ref{P_spectrum} without explanation. In particular, we identify the extension given by $(\mathcal{E},\theta)$ with an~element of $\Ext(\Omega)$. 

We shall now briefly explain how the lifting problem for operator semigroups is related to the lifting problem for separable abelian \cs-subalgebras of $\QQ(\Hh)$. We start with a~lemma saying that, apart from the continuity issue, our lifting problem is really about deciding whether the induced extension is the trivial one.

\begin{lemma}[{\bf `Lifting lemma'}]\label{lifting_L}
Let $X$ be an admissible compact metric space and let $\Gamma\in\Ext_{C_0,\gamma}(X)$ with respect to a~faithful $^\ast$-representation $\gamma\colon\QQ(\Hh)\to\BB(\mathbb{H})$. Suppose that the dyadic semigroup $(q(t))_{t\in\mathbb{D}}\subset\QQ(\Hh)$ induced by $\Gamma$ admits a~lift to a~dyadic $C_0$-semigroup of normal operators $(Q(t))_{t\in\mathbb{D}}\subset\BB(\Hh)$ {\rm (}i.e. $\pi Q(t)=q(t)$ for every $t\in\mathbb{D}${\rm )} such that the spectra of $q(2^{-n})$ and $Q(2^{-n})$ coincide for each $n\in\N$. Then $\Gamma=\Theta$.
\end{lemma}
\begin{proof}
Let $\sigma$ be the Busby invariant for $\Gamma$, $\p$ be the quotient map, and assume that the dyadic semigroup determined by the formula $q(2^{-n})=\sigma(\pi_n)$ ($n=0,1,2,\ldots$) admits a~lifting to a~$C_0$-semigroup $(Q(t))_{t\in\mathbb{D}}$. Write $X=\varprojlim X_n$ and let $\mathbf{A}\subset C(X)$ be the $^\ast$-subalgebra of functions which coincide with a~polynomial in variables $z$, $\oo{z}$ on some $X_n$, that is, $\mathbf{A}=\{P\circ\pi_n\colon P\in\C[z,\oo{z}],\,n\in\N\}$. Note that $\mathbf{A}$ is closed under addition and multiplication due to the fact that every polynomial acting on $X_n$ is also a~polynomial on $X_m$, for every $m>n$. Define $\rho$ on $\mathbf{A}$ by $\rho(f)=P(Q(2^{-n}),Q(2^{-n})^\ast)$ for $f=P\circ\pi_n$. Since addition, multiplication and involution in $\mathbf{A}$ reduce to the same operations on polynomials on some $X_n$, we see that $\rho\colon\mathbf{A}\to\BB(\Hh)$ is a~unital $^\ast$-homomorphism. Moreover, for any $f\in\mathbf{A}$ as above, $\|\rho(f)\|$ equals the supremum of $\vert P(z,\oo{z}\,)\vert$ over $z$ in the spectrum of $Q(2^{-n})$, which is the same as the spectrum of $q(2^{-n})$, and hence $\|\rho(f)\|=\sup\{\vert P(z,\oo{z}\,)\colon z\in X_n\}=\|f\|_\infty$, the supremum norm in $C(X)$. By the Stone--Weierstrass theorem, $\mathbf{A}$ is dense in $C(X)$, thus there exists a~unique continuous extension of $\rho$ to a~$^\ast$-homomorphism. Since for every $f\in\mathbf{A}$, $f=P\circ\pi_n$, we have
\begin{equation*}
\begin{split}
    \p\rho(f) &=\p\big[P(Q(2^{-n}),Q(2^{-n})^\ast)\big]=P\big[\p(Q(2^{-n})),\p(Q(2^{-n}))^\ast\big]\\
    &=P\big[\p\rho(\pi_n),\oo{\p(\rho(\pi_n)}\big]=P(\pi_n,\oo{\pi_n})=f,
\end{split}
\end{equation*}
we conclude that the obtained extension is a~right section for $\Gamma$. Hence, $\Gamma=\Theta$.
\end{proof}

\begin{example}\label{ex_2k}
Let $Z=\{2k\pi\ii\colon k\in\Z\}$ and 
$$
X_n=\oo{\exp(2^{-n}Z)}=\{\exp\small(2^{-n+1}k\pi\ii\small)\colon k=0,1,\ldots,2^n-1\},
$$
which is the set of all roots of unity of degree $2^n$, for $n=0,1,2,\ldots$ Let also $f_n\colon X_{n+1}\to X_n$ be given by $f(z)=z^2$. Then $X=\varprojlim\{X_n,f_n\}_{n\geq 0}$ is homeomorphic to the Cantor set, hence $\Ext(X)=0$. Consequently, any normal $C_0$-semigroup in $\QQ(\Hh)$ whose infinitesimal generator has spectrum $Z$ induces the zero extension.
\end{example}

\begin{example}\label{ex_imaginary}
Let $Z=\ii\R$ be the imaginary axis, so that $X_n=\exp(2^{-n}Z)=S^1$ for every $n=0,1,2,\ldots$ Let again each $f_n\colon S^1\to S^1$ be given by $f_n(z)=z^2$. Then, the inverse limit 
$$
\Sigma_2=\varprojlim\{S^1,z^2\}
$$
is the~$2$-{\it adic solenoid}. It is known (see \cite{KS}) that $\Ext(\Sigma_2)=0$. This can be proved, for example, by using Milnor's exact sequence (quoted in Theorem~\ref{milnor_thm}), which in this case has the form
$$
\begin{tikzcd}
0 \arrow[r] & \varprojlim^{(1)}\Ext(S^2) \arrow[r] & \Ext(\Sigma_2) \arrow[r] & \varprojlim \Z \arrow[r] & 0.
\end{tikzcd}
$$
By the well-known result \cite[Cor.~7.1]{BDF}, we have $\Ext(S^2)=0$, 
and since the connecting maps $\Z\to\Z$ in the quotient group are all given by $x\mapsto 2x$, we obviously have $\varprojlim\Z=0$. Therefore, $\Ext(\Sigma_2)$ is trivial and hence any normal $C_0$-semigroup in $\QQ(\Hh)$ whose generator has spectrum $\ii\R$ induces the zero extension.
\end{example}

We will return to this example in Section~4 to show the lack of \SOT-continuity, despite of the fact that all extensions of $\Sigma_2$ are trivial (see Example~\ref{T_not_cont_E}). Roughly speaking, the reason is that in the~infinite toin coss, both outcomes occur infinitely often almost surely.

\section{Proofs of the main results}

\subsection{Characterization of strong continuity}
In this final section we provide proofs of our main results. We start by characterizing strong continuity in terms of weak convergence in $\ell_\infty$.

\begin{proposition}\label{C0_i_iii_T}
Let $X$ be an admissible compact metric space, $X=\varprojlim X_n$, \vspace*{-1pt}where each $X_n$ is given by \eqref{adm_Z_D} and $Z\subset\C$ is a~closed set with 
$$
\eta\coloneqq\inf_{z\in Z}\mathrm{Re}\,z\leq\sup_{z\in Z}\mathrm{Re}\,z\eqqcolon\zeta<\infty.
$$
\begin{enumerate}[label={\rm (\roman*)}, leftmargin=24pt, labelindent=0pt, wide]
\setlength{\itemindent}{4pt}
\setlength{\itemsep}{9pt}
\item In order that $\Theta\in\Ext_{C_0,\gamma}(X)$ for a~fixed $\UU\!$-Calkin representation, it is necessary that $-\infty<\eta\leq\zeta<\infty$ and $\lim_{n\to\infty}\abs{1-\pi_n(\xi)}=0$ for every $\xi\in X$.

\item We have $\Theta\in\Ext_{C_0,\gamma}(X)$ for all $\mathscr{V}\!$-Calkin representations, with any $\mathscr{V}\in\beta\N\setminus\N$, if and only if for one (equivalently: for every) dense subset $\{\xi_k\colon k\in\N\}\subset X$, and any sequences $(L_n)_{n=1}^\infty$, $(S_n)_{n=1}^\infty$ of positive integers satisfying $\lim_{n\to\infty} 2^{-L_n}S_n=0$, we have
$$
\big(\vert 1-\pi_{L_n}(\xi_k)^{S_n}\vert\big)_{k=1}^\infty\yrightarrow[\,n\to\infty\,][2pt]{\,\,w\,\,}[0pt] 0\quad\mbox{in }\,\,\ell_\infty.
$$
\end{enumerate}
\end{proposition}

\begin{proof}
As we know, the zero element $\Theta$ is determined by a~Busby invariant $\tau=\pi\sigma$, where $\sigma\colon C(X)\to\BB(\Hh)$ is a~$^\ast$-homomorphism given by $\sigma(f)=\bigoplus_{k=1}^\infty f(\xi_k)I_{\Hh_k}$, where $\{\xi_k\colon k\in\N\}$ is an~arbitrary dense subset of $X$ and $\Hh\cong\bigoplus_{k=1}^\infty\Hh_k$ is a~decomposition into infinite-dimensional subspaces. Recall also that the dyadic semigroup $(q(t))_{t\in\mathbb{D}}$ generated by $\Theta$ is the extension of the semigroup $(q(2^{-n}))_{n\geq 0}$ given by $q(2^{-n})=\tau(\pi_n)$. 

Consider any $t\in\mathbb{D}$ of the form $t=2^{-\ell_1}+\ldots+2^{-\ell_k}$, where $1\leq \ell_1<\ldots<\ell_k$ are integers. Then
$$
q(t)=q(2^{-\ell_1})\cdot\ldots\cdot q(2^{-\ell_k})=\tau(\pi_{\ell_1}\cdot\ldots\cdot\pi_{\ell_k})
$$
and since
\begin{equation}\label{piii}
\pi_{\ell_i}^{2^{\ell_i-\ell_1}}=\pi_{\ell_1}\quad\,\mbox{for each }\,\,1\leq i\leq k,
\end{equation}
we can write $q(t)=\tau(\pi_{\ell_k}^s)$, where $s=\sum_{i=1}^k 2^{\ell_i-\ell_1}$. Now, let $(t_n)_{n=1}^\infty\subset\mathbb{D}$ be any sequence converging to zero, 
$$
t_n=2^{-\ell_{n,1}}+\ldots+2^{-\ell_{n,k_n}}\quad (n\in\N)
$$
with integers $1\leq \ell_{n,1}<\ldots<\ell_{n,k_n}$, and define: 
\begin{equation}\label{FLS_def}
F_n=\ell_{n,1},\,\,\,\, L_n=\ell_{n,k_n},\,\,\,\, S_n=\sum_{i=1}^{k_n} 2^{\ell_{n,i}-F_n}\quad\, (n\in\N).
\end{equation}
Note that $t_n\to 0$ if and only if $F_n\to \infty$. Observe also that $S_n=1+\sum_{m\in M}2^m$ for some $M\subseteq\{1,2,\ldots,L_n-F_n\}$, hence $1\leq S_n\leq 2^{L_n-F_n+1}-1$ from which it follows that $2^{-L_n}S_n\to 0$ as $n\to\infty$. Applying \eqref{piii} to $t_n$ we obtain $q(t_n)=\tau(\pi_{L_n}^{S_n})$ for each $n\in\N$. Therefore, for any $\mathbf{x}=[(x_j)]_\sim\in\mathscr{W}^{\usim}$ and $n\in\N$, we have
\begin{equation*}
    \begin{split}
        \mathbf{x}-\gamma(q(t_n))\mathbf{x} &=\mathbf{x}-\gamma\Bigg\{\bigoplus_{k=1}^\infty\pi_{L_n}(\xi_k)^{S_n} I_{\Hh_k}\Bigg\}\mathbf{x}\\[2ex]
        &=\Bigg[\Big(\bigoplus_{k=1}^\infty\big(1-\pi_{L_n}(\xi_k)^{S_n}\big) I_{\Hh_k}x_j\Big)_{\!\!j=1}^{\!\!\infty}\Bigg]_\sim\\[2ex]
        &=\Bigg[\Big(\sum_{k=1}^\infty\big(1-\pi_{L_n}(\xi_k)^{S_n}\big)P_kx_j\Big)_{\!\!j=1}^{\!\!\infty}\Bigg]_\sim,
    \end{split}
\end{equation*}
where $P_k$ is the orthogonal projection onto $\Hh_k$. Set $\omega_{j,k}=\n{P_kx_j}^2$ and notice that if $(x_j)$ runs through the collection of all weakly null sequences in $\Hh$, then $\oo{\mathbf{\omega}}=(\omega_{j,k})_{j,k\in\N}$ is a~matrix of nonnegative entries such that $\sup_j\sum_{k=1}^\infty \omega_{j,k}<\infty$, as $\sum_{k=1}^\infty\omega_{j,k}=\sum_{k=1}^\infty\n{P_kx_j}^2=\n{x_j}^2$ for every $j\in\N$. Conversely, since all the $\Hh_k$ are infinite-dimensional, every such matrix corresponds to some weakly null sequence $(x_j)\subset\Hh$. Therefore, the condition
\begin{equation}\label{condition_x}
    \lim_{n\to\infty}\|\mathbf{x}-\gamma(q(t_n))\mathbf{x}\|_{\mathbb{H}}=0\quad\mbox{for every }\,\,\mathbf{x}\in\mathscr{W}^{\usim}
\end{equation}
is equivalent to saying that
\begin{equation}\label{condition_LIM}
    \lim_{n\to\infty}\mathop{\mathrm{LIM}}_{j,\UU}\,\sum_{k=1}^\infty\omega_{j,k}\vert 1-\pi_{L_n}(\xi_k)^{S_n}\vert^2=0
\end{equation}
for every matrix $\oo{\omega}$ as described above. 

For any $n\in\N$, the fixed dense set $\{\xi_k\colon k\in\N\}\subset X$ and $(t_n)\subset\mathbb{D}$ with $t_n\to 0$, define 
\begin{equation}\label{xi_def}
\boldsymbol{\xi}^{(n)}=\big(\vert 1-\pi_{L_n}(\xi_k)^{S_n}\vert\big)_{k=1}^\infty\in\ell_\infty;
\end{equation}
notice that this is indeed a~bounded sequence, as
\begin{equation}\label{est_l}
    \begin{split}
        \n{\boldsymbol{\xi}^{(n)}}_\infty &=\sup_{k\in\N}\vert 1-\pi_{L_n}(\xi_k)^{S_n}\vert\\
        &=\sup\big\{\vert 1-z^{S_n}\vert\colon z\in\exp(2^{-L_n}Z)\big\}\leq 1+\exp(2^{-L_n}S_n\zeta).
    \end{split}
\end{equation}
Putting $\omega_{j,k}=c_k$ for all $j,k\in\N$ in condition \eqref{condition_LIM}, where $c_k\geq 0$ and $\sum_{k=1}^\infty c_k=1$, \vspace*{-2pt}we obtain $\sum_{k=1}^\infty c_k(\boldsymbol{\xi}^{(n)}_k)^2\to 0$ as $n\to\infty$ from which it follows that $(\boldsymbol{\xi}^{(n)})_{n=1}^\infty$ converges to zero in the weak$^\ast$ topology on $\ell_\infty$. Note that if $\eta=-\infty$, we would have $0\in X_n$ for each $n\geq 0$, hence $\boldsymbol{0}=(0,0,\ldots)\in X$ and taking e.g. $\xi_1=\boldsymbol{0}$ we see that the $(\boldsymbol{\xi}^{(n)})_{n=1}^\infty$ would not converge weak$^\ast$ to zero, hence we must have $\eta>-\infty$. Finally, for any $\zeta<\infty$, \eqref{est_l} shows that $\sup_{n\in\N}\n{\boldsymbol{\xi}^{(n)}}_\infty$ is finite, thus the weak$^\ast$ convergence of this sequence is equivalent to coordinatewise convergence. This finishes the proof of assertion (i).

\vspace*{2mm}
Now, we shall prove assertion (ii). According to Lemma~\ref{cont1_L}, we have $\Theta\in\Ext_{C_0,\gamma}(X)$ under a~$\mathscr{V}\!$-Calkin representation $\gamma$ if and only if condition \eqref{condition_x} is valid for every sequence $(t_n)_{n=1}^\infty\subset\mathbb{D}$ with $t_n\to 0$. This is in turn equivalent to condition \eqref{condition_LIM} being valid for:
\begin{itemize}[leftmargin=24pt]
\setlength{\itemsep}{2pt}
\item every $\mathscr{V}\in\beta\N\setminus\N$ in place of $\UU$, 

\item every matrix $\oo{\omega}=(\omega_{j,k})_{j,k\in\N}$ with nonnegative entries such that $\sup_j\sum_{k=1}^\infty \omega_{j,k}<\infty$,

\item and every dense set $\{\xi_k\colon k\in\N\}\subset X$.
\end{itemize}
Of course, it does not depend on the particular choice of $\{\xi_k\colon k\in\N\}$, as we know that the property of inducing a~\SOT-continuous semigroup is preserved by taking an~equivalent extension (see Remark~\ref{C0_correct_R}). Moreover, the first two quantifiers can be replaced by saying simply that 
\begin{equation}\label{for_every_V}
\lim_{n\to\infty}\mathop{\mathrm{LIM}}_{k,\mathscr{V}}\,\boldsymbol{\xi}^{(n)}_k=0\quad\mbox{for every }\,\,\mathscr{V}\in\beta\N.
\end{equation}
To see this, note that for every choice of $\mathscr{V}$ and $\oo{\omega}$, the map $\ell_\infty \ni\boldsymbol{\xi}\mapsto\mathop{\mathrm{LIM}}_{j,\mathscr{V}}\sum_{k=1}^\infty \omega_{j,k}\boldsymbol{\xi}_k$ is an element of $\ell_\infty^\ast$, therefore \eqref{condition_LIM} assumed for all choices of $\mathscr{V}$ and $\oo{\omega}$ is (formally) weaker than saying that $(\boldsymbol{\xi}^{(n)})_{n=1}^\infty$ converges weakly to zero (here, we can obviously omit the square). Taking $\omega_{j,k}=1$ for any $j=k\in\N$ and $\omega_{j,k}=0$ otherwise, we see that \eqref{condition_LIM} implies \eqref{for_every_V} (we can include $\mathscr{V}\in\N$ in that condition, as we have already observed that \eqref{condition_LIM} implies the coordinatewise convergence). On the other hand, assuming \eqref{for_every_V} and using the 
fact that boundedness and pointwise convergence of sequences in $C(K)$-spaces\vspace*{-1pt} implies weak convergence (see, e.g., \cite[Cor.~3.138]{FHHMZ}), we conclude\vspace*{-1pt} that since $\lim_{n\to\infty}\boldsymbol{\xi}^{(n)}(\mathscr{V})$ for every $\mathscr{V}\in\beta\N$ and $(\boldsymbol{\xi}^{(n)})_{n=1}^\infty$ is uniformly bounded, \eqref{for_every_V} yields $\boldsymbol{\xi}^{(n)}\xrightarrow[]{w}0$.

Recall that the parameters $L_n$ and $S_n$ were defined in terms of $t_n\in\mathbb{D}$ and since $(t_n)_{n=1}^\infty$ was an~arbitrary sequence of positive dyadic numbers converging to zero, it is easily seen that condition \eqref{for_every_V} must be valid for $(\boldsymbol{\xi}^{(n)})_{n=1}^\infty$ defined by \eqref{xi_def} with arbitrary sequences $(L_{n})_{n=1}^\infty, (S_n)_{n=1}^\infty\subset\N$ satisfying $2^{-L_n}S_n\to 0$. Indeed, given $(L_n)_{n=1}^\infty$ and $(S_n)_{n=1}^\infty$ as above, formulas \eqref{FLS_def} uniquely determine a~sequence $(t_n)_{n=1}^\infty\subset\mathbb{D}$ with $2^{-L_n}S_n\geq 2^{-F_n}$, hence $F_n\to\infty$ which means that $t_n\to 0$.

Conversely, if for all sequences $(L_{n})_{n=1}^\infty, (S_n)_{n=1}^\infty\subset\N$ satisfying $2^{-L_n}S_n\to 0$ and any dense set $\{\xi_k\colon k\in\N\}\subset X$, we have $\boldsymbol{\xi}^{(n)}\xrightarrow[]{w}0$, then \eqref{condition_x} holds true for every $(t_n)_{n=1}^\infty\subset\mathbb{D}$ and $\gamma$ being any $\mathscr{V}\!$-Calkin representation, which, as we know from Lemma~\ref{cont1_L}, guarantees that $\Theta\in\Ext_{C_0,\gamma}(X)$. This finishes the proof of assertion (ii).
\end{proof}

It is known that weak convergence in $B(\Sigma)$-spaces, that is, Banach spaces of bounded measurable functions, where $\Sigma$ is a~$\sigma$-algebra on some set $A$, can be characterized in terms of quasi-uniform convergence; see \cite[\S VI.6]{DS}. Namely, a~sequence $(f_n)_{n=1}^\infty\subset B(\Sigma)$ is weakly null if and only if it is bounded, pointwise convergent to zero and every its subsequence $(f_{n_k})_{k=1}^\infty$ satisfies the following condition: for each $\e>0$ and $k_0\in\N$ there are indices $k_0\leq k_1<\ldots<k_j$ such that $\min_{1\leq i\leq j}\vert f_{n_{k_i}}(a)\vert<\e$ for every $a\in A$ (see \cite[Thm.~VI.6.31]{DS}). We can therefore reformulate the characterization in assertion (ii) into a~more directly applicable form.

For the rest of this section we assume that $X$ is an~admissible compact metric space as described in Proposition~\ref{C0_i_iii_T}.

\begin{corollary}\label{DS_corollary}
We have $\Theta\in\Ext_{C_0,\gamma}(X)$ for all $\mathscr{V}\!$-Calkin representations, with an~arbitrary $\mathscr{V}\in\beta\N\setminus\N$, if and only if the following two conditions hold true:
\begin{itemize}[leftmargin=24pt]
\setlength{\itemsep}{3pt}
\item[{\rm (a)}] $\lim_{n\to\infty}\vert 1-\pi_n(\xi)\vert=0$ for every $\xi\in X$;

\item[{\rm (b)}] for any sequences $(L_n)_{n=1}^\infty$, $(S_n)_{n=1}^\infty$ of positive integers satisfying $\lim_{n\to\infty} 2^{-L_n}S_n=0$, and $\e>0$ and $n_0\in\N$, there exist indices $n_0\leq n_1<\ldots<n_k$ such that
\begin{equation}\label{min_e}
\min\limits_{1\leq i\leq k}\vert 1-\pi_{L_{n_i}}(\xi)^{S_{n_i}}\vert<\e\quad\mbox{for every }\,\, \xi\in X.
\end{equation}
\end{itemize}
\end{corollary}

\begin{proof}
The necessity follows directly from Proposition~\ref{C0_i_iii_T}(ii) which yields \eqref{min_e} for $\xi_k$ from a~dense subset of $X$ in place of $\xi$, but this, of course, implies that the estimate is valid for every $\xi\in X$. For sufficiency, recall that by \eqref{est_l}, the assumption $\zeta=\sup_{z\in Z}\mathrm{Re}\,z<\infty$ implies that for any dense set $\{\xi_k\colon k\in\N\}\subset X$ and any $(L_n)_{n=1}^\infty$, $(S_n)_{n=1}^\infty$ \vspace*{-1pt}as above, the sequence $(\boldsymbol{\xi}^{(n)})_{n=1}^\infty\subset\ell_\infty$ given by \eqref{xi_def} is bounded. Also, since we allow $(L_n)_{n=1}^\infty$ and $(S_n)_{n=1}^\infty$ to be arbitrary sequences satisfying $2^{-L_n}S_n\to 0$, every subsequence of $(\boldsymbol{\xi}^{(n)})_{n=1}^\infty$ has the property of quasi-uniform convergence. Therefore, it is weakly convergent to zero due to the above quoted characterization \cite[Thm.~VI.6.31]{DS}. The~result now follows from Proposition~\ref{C0_i_iii_T}(ii).
\end{proof}

We are almost ready to prove Theorem A; we just need one more technical lemma.
\begin{lemma}\label{rate_L}
Suppose that for some $\xi \in X$, we have $\lim_{n\to\infty}\pi_n(\xi)=1$. Then
$$
\abs{1-\pi_n(\xi)}=O(2^{-n}).
$$
Hence, if condition {\rm (a)} of Corollary~\ref{DS_corollary} holds true, then for every $\xi\in X$ there is a~constant $C=C(\xi)>0$ such that $\abs{1-\pi_n(\xi)}\leq C 2^{-n}$ for $n\in\N$. 
\end{lemma}

\begin{proof}
First, note that our assumption implies that $\eta=\inf_{z\in Z} \mathrm{Re}\,z>-\infty$ as otherwise $(0,0,\ldots)\in X$. Since the connecting maps $f_n\colon X_{n+1}\to X_n$ are given by $f_n(z)=z^2$, we immediately see that if $\pi_n(\xi)\to 1$, then there must exists $n_0\in\N$ such that for every $n\geq n_0$, $\pi_{n+1}(\xi)=\sqrt{\pi_n(\xi)}$ is given by the principal branch of square root. In other words, 
$$
\pi_{n+1}(\xi)=\sqrt{\vert \pi_n(\xi)\vert}\, e^{\ii\mathrm{Arg} (\pi_n(\xi))/2} \quad (n\geq n_0),
$$
hence, 
$$
\abs{\pi_{n+k}(\xi)}=\abs{\pi_n(\xi)}^{2^{-k}}\quad\mbox{and}\quad \mathrm{Arg}\,\pi_{n+k}(\xi)=2^{-k}\mathrm{Arg}\,\pi_n(\xi) \quad (n\geq n_0,\, k\in\N).
$$
Since for any $m\in\N$ we have $\abs{\pi_m(\xi)}\leq\max\{1,\exp(\zeta)\}$, it follows that with suitable constants $C_1, C_2>0$ we have 
$$2^{-k}C_1\geq 2^{-k}\mathrm{log}\,\abs{\pi_n(\xi)}=\mathrm{log}\,\abs{\pi_{n+k}(\xi)}\geq \frac{\abs{\pi_{n+k}(\xi)}-1}{\abs{\pi_{n+k}(\xi)}}\geq C_2(\abs{\pi_{n+k}(\xi)}-1)
$$
Similarly, if $\abs{\pi_{n}(\xi)}<1$, then using the inequality $\abs{\pi_m(\xi)}\geq\min\{1,\exp(\eta)\}$ we obtain
$$
2^{-k}C_3\geq -2^{-k}\log\,\abs{\pi_n(\xi)}=-\log\,\abs{\pi_{n+k}(\xi)}\geq 1-\abs{\pi_{n+k}(\xi)}.
$$
with some $C_3>0$. Therefore, there is a~constant $C>0$ such that $\vert\abs{\pi_{n+k}(\xi)}-1\vert\leq C2^{-k}$ for all $n\geq n_0$, $k\in\N$. Therefore, for $n=n_0$ and any $k\in\N$, we have 
\begin{equation*}
\begin{split}
\abs{1-\pi_{n_0+k}(\xi)} &\leq \big|1-\abs{\pi_{n_0+k}(\xi)}\big|+\abs{\pi_{n_0+k}(\xi)}\!\cdot\!\big|1-\exp(\ii\mathrm{Arg}\,(\pi_{n_0+k}(\xi)))\big|\\
& \leq \big(C+\max\{1,\exp(\zeta)\}\!\cdot\!\abs{\mathrm{Arg}\,\pi_{n_0}(\xi)}\big)\cdot 2^{-k},
\end{split}
\end{equation*}
which proves our assertion.
\end{proof}

\begin{proof}[Proof of Theorem A]
In view of Proposition~\ref{C0_i_iii_T} and Corollary~ \ref{DS_corollary}, it suffices to show that the pointwise convergence of $(\pi_n)_{n=0}^\infty$ to $\mathbf{1}$ implies condition (b) of Corollary~\ref{DS_corollary}. Fix $\e>0$ and any sequences $(L_n)_{n=1}^\infty$ and $(S_n)_{n=1}^\infty$ of positive integers such that $2^{-L_n}S_n\to 0$. Define
$$
U_j=\big\{\xi\in X\colon\abs{1-\pi_{L_j}(\xi)^{S_j}}<\e\big\}\quad\,\, (j\in\N).
$$
Notice that since
$$
\abs{1-\pi_{L_j}(\xi)^{S_j}}=\abs{1-\pi_{L_j}(\xi)}\!\cdot\!\abs{1+\pi_{L_j}(\xi)+\ldots+\pi_{L_j}(\xi)^{S_j-1}}
$$
and $\vert\pi_{L_j}(\xi)\vert\leq\max\{1,\exp(2^{-L_j}\zeta)\}$, we have 
$$
\abs{1-\pi_{L_j}(\xi)^{S_j}}\leq\left\{\begin{array}{ll}
\abs{1-\pi_{L_j}(\xi)}\!\cdot\! S_j & \mbox{if }\,\,\zeta\leq 0\\[4pt]
\displaystyle{\abs{1-\pi_{L_j}(\xi)}\!\cdot\!\frac{\exp(2^{-L_j}S_j\zeta)-1}{\exp(2^{-L_j}\zeta)-1}} & \mbox{if }\,\,\zeta> 0.
\end{array}\right.
$$
Hence, by Lemma~\ref{rate_L}, in both cases we have 
$$
\abs{1-\pi_{L_j}(\xi)^{S_j}}=O(2^{-L_j}S_j)\xrightarrow[j\to\infty]{}0.
$$
Therefore, $\ccup_{j\geq n_0} U_j=X$ and there exists a~finite subcover $\{U_{n_1},\ldots, U_{n_k}\}$ of $X$, where $n_0\leq n_1<\ldots<n_k$. This means exactly that condition \eqref{min_e} is satisfied.
\end{proof}

\subsection{Automatic continuity}
\begin{lemma}\label{point_unif_L}
Let $X=\varprojlim X_n$ be an admissible compact metric space with $X_n$ given\vspace*{-1pt} by \eqref{adm_Z_D}, where $Z\subset\C$ is a~closed set with $\zeta\coloneqq \sup_{z\in Z}\mathrm{Re}\, z<\infty$ such that conditions {\rm (H$_1$)} and {\rm (H$_2$)} hold true. Then, $\pi_n\yrightarrow[][2pt]{}[2pt]\mathbf{1}$ pointwise on $X$ if and only if $\pi_n\yrightarrow[][0pt]{}[0pt]\mathbf{1}$ uniformly on $X$.
\end{lemma}
\begin{proof}
Assume that $\pi_n(\xi)\to 1$ as $n\to\infty$ for every $\xi\in X$. As we have observed earlier, this implies that $\eta\coloneqq\inf_{z\in Z}\mathrm{Re}\, z>-\infty$. Therefore,
$$
e^{2^{-n}\eta}\leq\abs{z}\leq e^{2^{-n}\zeta}\quad (n=0,1,2,\ldots, \,\,\, z\in X_n),
$$
which shows that $\abs{\pi_n(\xi)}\to 1$ uniformly for $\xi\in X$. Hence, it suffices to show that $\mathrm{Arg}\,\pi_n(\xi)$ converges to zero uniformly for $\xi\in X$ as $n\to\infty$. 

To this end, we introduce the following definitions. First, recall that for any $t\in [\eta,\zeta]$ we have defined the vertical section $S_t$ of $Z$ by $S_t=\{u\in\R\colon t+\ii u\in Z\}$. We also define the $n^{\mathrm{th}}$ circle section corresponding to $t$ by
$$
\mathcal{C}_n(t)=X_n\cap\big\{\abs{z}=e^{2^{-n}t}\big\}.
$$
Notice that $
\exp(2^{-n}S_t)\subseteq\mathcal{C}_n(t)$, and that this inclusion may be strict for a~general closed set $Z\subset\C$, as there may be points in $Z$ participating in the set $\mathcal{C}_n(t)$ other than those for which $\mathrm{Re}\,z=t$. However, by assumption (H$_1$), $\exp(2^{-n}Z)$ is closed, thus $\mathcal{C}_n(t)=\exp(2^{-n}S_t)$ for all $t\in [\eta,\zeta]$ and all $n$ sufficiently large, say $n\geq n_0$.

For any set $B\subset\C$, denote by $\mathsf{A}(B)$ the set of points that belong to $B$ together with its antipode, i.e. $\mathsf{A}(B)=\{z\in B\colon -z\in B\}$. Of course, $\mathsf{A}(B)$ is closed, provided that $B$ is closed.

\vspace*{2mm}\noindent
\underline{{\it Claim.}} Let $t\in [\eta,\zeta]$ and $n\in\N$, $n\geq n_0$. We have $\mathsf{A}(\mathcal{C}_n(t))\neq\varnothing$ if and only if
$$
\big(S_t+2^{n+1}\pi\Z\big)\cap\big(S_t-2^n\pi+2^{n+1}\pi\Z   \big)\neq\varnothing.\leqno(\star)
$$

\vspace*{2mm}\noindent
Indeed, $\mathsf{A}(\mathcal{C}_n(t))\neq\varnothing$ if and only if there is $\theta\in\R$ such that both numbers $\exp(2^{-n}t+2^{-n}\ii\theta)$ and $\exp(2^{-n}t+(2^{-n}\theta+\pi)\ii)$ of modulus $e^{2^{-n}t}$ belong to $X_n=\exp(2^{-n}Z)$. This is in turn equivalent to the existence of two elements $t+\ii u, t+\ii v\in Z$ with 
$$
\left\{\begin{array}{rcl}
\exp(2^{-n}t+2^{-n}\ii\theta) & = & \exp(2^{-n}(t+\ii u))\\ \exp(2^{-n}t+(2^{-n}\theta+\pi)\ii) & = & \exp(2^{-n}(t+\ii v)).
\end{array}\right.
$$
Therefore, for some $k,\ell\in\Z$ we must have $\theta=u+2^{n+1}k\pi$ and $\theta+2^n\pi=v+2^{n+1}\ell\pi$. This is plainly equivalent to saying that $\theta$ belongs to the set at the left-hand side of ($\star$), as $u,v\in S_t$.

\vspace*{2mm}
In view of our Claim, we have $\mathsf{A}(\mathcal{C}_n(t))\neq\varnothing$ if and only if $n\in M_t$. By assumption (H$_2$), we infer that there is $n_1\in\N$ such that $\mathsf{A}(\mathcal{C}_n(t))=\varnothing$ for all $n\geq n_1$ and $t\in [\eta,\zeta]$. Since there are no pairs of antipodal points in any of the $X_n$ ($n\geq n_1$), we obtain $\mathrm{Arg}\,\pi_n(\xi)\to 0$ uniformly for $\xi\in X$, which concludes the proof.
\end{proof}

\begin{proof}[Proof of Theorem B]
Since $\Ext(X)=0$, the extension generated by $(q(t))_{t\geq 0}$ by means of Proposition~\ref{P_spectrum} is the trivial one. By Theorem~A and Lemma~\ref{point_unif_L}, we have $\pi_n\to\mathbf{1}$ uniformly on $X$. The corresponding Busby invariant $\tau\colon C(X)\to\QQ(\Hh)$ satisfies $\tau(\pi_n)=q(2^{-n})$ for $n\in\N$, thus $\n{q(2^{-n})-I}\to 0$ as $n\to\infty$ which implies that $(q(t))_{t\geq 0}$ is uniformly continuous.
\end{proof}

\begin{proof}[Proof of Theorem C]
Plainly, the dyadic semigroup $(q(t))_{t\in\Dd}$ has a~lift to the $C_0$-semigroup $(Q(t))_{t\in\Dd}$ of normal operators with the generator $T$. Moreover, $\sigma(q(2^{-n}))=\sigma_{\mathrm{ess}}(Q(2^{-n}))$ for each $n\in\N$. By the assumption, $\sigma(T)=\sigma_{\mathrm{ess}}(T)$ and hence, by Nagy's spectral mapping theorem (see \cite{nagy}), we have $\sigma (q(2^{-n}))=\sigma (Q(2^{-n}))$ for each $n\in\N$. Therefore, the `lifting lemma' (Lemma~\ref{lifting_L}) yields that the extension induced by $(q(t))_{t\geq 0}$ is the zero extension. Again, by Theorem~A and Lemma~\ref{point_unif_L}, we have $\pi_n\to\mathbf{1}$ uniformly on $X$, which implies that $(q(t))_{t\geq 0}$ is uniformly continuous.
\end{proof}

\begin{example}
Let $Z=\{s+\ii\hspace*{1pt} t\colon (s,t)\in [-1,0]\times [-\pi,\pi]\}$. Then $X_n=\exp\small(2^{-n}Z\small)$ is the set of complex numbers $z$ with $e^{-2^{-n}}\leq\abs{z}\leq 1$ and $-2^{-n}\pi\leq \mathrm{Arg}\,z\leq 2^{-n}\pi$. Hence, for any $\xi\in X$ we have
$$
\vert 1-\pi_n(\xi)\vert=\max\big\{\vert 1-e^{2^{-n}\pi\ii}\vert,\,\vert 1-e^{-2^{-n}(1+\pi\ii)}\vert\big\}
$$
which for $n$ large enough (so that $2^{-n}$ is smaller than the smallest positive root of the equation $1+e^{-x}=2\cos\pi x$) equals 
$$
|1-e^{-2^{-n}(1+\pi\ii)}|=\sqrt{(1-e^{-2^{-n}})^2+2e^{-2^{-n}}(1-\cos 2^{-n}\pi)}=O(2^{-n}).
$$
Obviously, it follows from this example that for an~arbitrary compact set $Z\subset\C$, the semigroup $(q(t))_{t\in\Dd}$ induced by $X=\varprojlim X_n$ is uniformly continuous.
\end{example}

\begin{example}\label{T_not_cont_E}
As in Example~\ref{ex_imaginary}, let $Z=\ii\R$ and $X_n=\exp(2^{-n}Z)=S^1$ for $n=0,1,2,\ldots$ Notice that each element $\xi$ of $\Sigma_2=\varprojlim\{S^1,z^2\}$ is uniquely determined \vspace*{-1pt}by a~choice of $\pi_0(\xi)\in S^1$ and a~sequence $(\e_n)_{n=1}^\infty\in\{0,1\}^\N$, in the sense that for each $n\in\N$, if $\pi_{n-1}(\xi)=e^{it}$, then $\pi_{n}(\xi)=e^{\ii(t+2\e_n\pi)/2}$. Obviously, for almost all choices of $(\e_i)_{i=1}^\infty$, the sequence $(\pi_n(\xi))_{n=0}^\infty$ diverges, thus the pointwise convergence condition in Proposition~\ref{C0_i_iii_T}(i) is not satisfied. Hence, we have $\Theta\not\in\Ext_{C_0,\gamma}(\Sigma_2)$.
\end{example}

\begin{example}\label{c0not}
Take any sequence $(n_j)\subset\N$ with $n_j\nearrow\infty$ and let
$$
\alpha_j=\frac{\pi}{j}+2^{n_j+1}\pi,\quad \beta_j=\frac{\pi}{j}+3\cdot 2^{n_j}.
$$
Denote by $\mathbb{P}$ the set of primes and define 
$$
Z=\{\ii\alpha_j\colon j\in\mathbb{P},\, j\geq 3\}\cup \{\ii\beta_j\colon j\in\mathbb{P},\, j\geq 3\}.
$$
We claim that the sets of antipodal pairs in $X_N=\exp(2^{-N}Z)$ are given by
\begin{equation}\label{AS}
\mathsf{A}(X_{n_j})=\big\{\pm\exp(2^{-n_j}\alpha_j\ii)\big\}\quad\mbox{for }j\in\mathbb{P},\, j\geq 3
\end{equation}
and 
\begin{equation}\label{ASB}
\mathsf{A}(X_N)=\varnothing\quad\mbox{for }\,N\not\in\{n_2,n_3,n_5,\ldots\}.
\end{equation}
First, one easily verifies that the equality $\exp(2^{-N}\alpha_j\ii)=-\exp(2^{-N}\alpha_k\ii)$ is impossible for $j\neq k$. Similarly, if $\exp(2^{-N}\alpha_j\ii)=-\exp(2^{-N}\beta_k\ii)$, then $\tfrac{1}{j}+2^{n_j+1}=\tfrac{1}{k}+3\cdot 2^{n_k}+2^N+2^{N+1}r$ for some $r\in\Z$, thus $j=k$ and $2^{n_k}+2^N(1+2r)=0$ which implies $n_k=N$ and $r=-1$.

Observe also that two different angles $\alpha_j$, $\alpha_k$ cannot participate in a~single element $\xi\in X=\varprojlim X_N$. Indeed, if we had $\exp(2^{-M}\alpha_j\ii)=\exp(2^{-N}\alpha_k\ii)$, then 
$$
\frac{1}{j}+2^{n_j+1}=\frac{2^{M-N}}{k}+2^{n_k+1}+2^{M+1}r
$$
with $r\in\Z$ and since $j, k$ are prime numbers, we would have $j=k$ and $M=N$. Therefore \eqref{AS} and
\eqref{ASB} imply that $\pi_n\to\mathbf{1}$ pointwise but not uniformly on $X$, which means that the semigroup $(q(t))\subset \QQ(\Hh)$ induced by $X$ is strongly continuous with respect to all Calkin's representations, but is not uniformly continuous.
\end{example}

\vspace*{2mm}\noindent
{\bf Acknowledgement. }I acknowledge with gratitude the support from the National Science Centre, grant OPUS 19, project no.~2020/37/B/ST1/01052.

\bibliographystyle{amsplain}

\end{document}